\documentclass[11pt]{article}
\usepackage{graphicx}
\usepackage{tabularx}
\usepackage{url}
\usepackage{amsthm} 
\usepackage{amsmath} 
\usepackage{amsfonts}

\theoremstyle{plain}
\newtheorem{thm}{Theorem}
\newtheorem{lem}[thm]{Lemma}

\newtheorem{claim}[thm]{Claim}
\newtheorem{cor}[thm]{Corollary}
\theoremstyle{definition}
\newtheorem{dfn}{Definition}

\theoremstyle{remark}
\newtheorem*{rem}{Remark}

\begin{document}
\title{Bipartite Induced Subgraphs and Well-Quasi-Ordering\thanks{This 
research was supported by DIMAP -- the Center for Discrete Mathematics 
and its Applications at the University of Warwick}} 

\author{Nicholas Korpelainen\thanks{DIMAP and 
Mathematics Institute, University of Warwick, Coventry, CV4 7AL, UK. 
E-mail: N.Korpelainen@warwick.ac.uk}\and Vadim V. Lozin\thanks{DIMAP 
and Mathematics Institute, University of Warwick, Coventry, CV4 7AL, UK. 
E-mail: V.Lozin@warwick.ac.uk}}
\date{}
\maketitle
 
\begin{abstract}
We study bipartite graphs partially ordered by the induced subgraph relation.
Our goal is to distinguish classes of bipartite graphs which are or are not 
well-quasi-ordered (wqo) by this relation. Answering an open question from \cite{Ding92}, 
we prove that $P_7$-free bipartite graphs are not wqo. On the other hand, we show that 
$P_6$-free bipartite graphs are wqo. We also obtain some partial results on subclasses 
of bipartite graphs defined by forbidding more than one induced subgraph.
\end{abstract}


\section{Introduction}
A binary relation $\le$ on a set $X$ is a {\it quasi-order} if it is reflexive and transitive.
Two elements $x,y\in X$ are said to be incomparable if neither $x\le y$ nor $y\le x$. 
An {\it antichain} in a quasi-order is a set of pairwise incomparable elements.   
A quasi-order $(X,\le)$ is a {\it well-quasi-order} if $X$
contains no infinite strictly decreasing sequences and no infinite antichains. 

In this paper, we study binary relations defined on sets of graphs. A graph $H$ is said to be a {\it minor} 
of a graph $G$ if $H$ can be obtained from $G$ by a (possibly empty) sequence of vertex deletions, 
edge deletions and edge contractions. According to the celebrated Graph Minor Theorem of Robertson and Seymour, 
the set of all graphs is well-quasi-ordered by the graph minor relation \cite{RS88}. 
This, however, is not the case for 
the more restrictive relations such as subgraphs or induced subgraphs. A graph $H$ is a {\it subgraph} 
of $G$ if $H$ can be obtained from $G$ by a (possibly empty) sequence of vertex deletions
and edge deletions; $H$ is an {\it induced subgraph} of $G$ if $H$ can be obtained from $G$ 
by a (possibly empty) sequence of vertex deletions. Clearly, the cycles $C_3,C_4,C_5,\ldots$ form 
an infinite antichain with respect to both relations. Except for this example, only a few other infinite antichains
are known with respect to the subgraph or induced subgraph relations. One of them is the sequence of graphs 
$H_1,H_2,H_3,\ldots$ represented in Figure~\ref{fig:H}(left). 
Moreover, Ding proved in \cite{Ding92} that, in a sense, the cycles $C_3,C_4,C_5,\ldots$ and the graphs $H_1,H_2,H_3,\ldots$
are the only two infinite antichains with respect to the subgraph relation. More formally, Ding proved that 
a class of graphs closed under taking subgraphs is well-quasi-ordered by the subgraph relation if and only if it 
contains finitely many graphs $C_n$ and $H_n$. The situation with induced subgraphs is less explored.

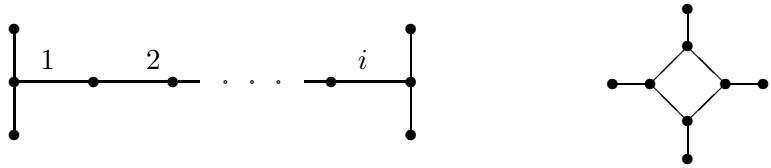
\begin{figure}[ht]
\begin{center}
\begin{picture}(170,80)
\put(-10,25){\circle*{4}} 
\put(20,25){\circle*{4}}
\put(50,25){\circle*{4}} 
\put(70,25){\circle{1}}
\put(80,25){\circle{1}} 
\put(90,25){\circle{1}}
\put(110,25){\circle*{4}} 
\put(140,25){\circle*{4}}
\put(-10,45){\circle*{4}} 
\put(-10,5){\circle*{4}}
\put(140,45){\circle*{4}} 
\put(140,5){\circle*{4}}
\put(-8,25){\line(1,0){26}} 
\put(22,25){\line(1,0){26}}
\put(52,25){\line(1,0){8}} 
\put(100,25){\line(1,0){8}}
\put(112,25){\line(1,0){26}} 
\put(-10,27){\line(0,1){16}}
\put(-10,23){\line(0,-1){16}} 
\put(140,27){\line(0,1){16}}
\put(140,23){\line(0,-1){16}} 
\put(0,30){1} \put(40,30){2}
\put(120,30){$i$}
\end{picture}
\begin{picture}(130,60)
\setlength{\unitlength}{0.5mm}
\put(30,17){\circle*{3}}
\put(40,17){\circle*{3}} 
\put(50,-3){\circle*{3}} 
\put(50,7){\circle*{3}}
\put(50,27){\circle*{3}} 
\put(50,37){\circle*{3}}
\put(60,17){\circle*{3}}
\put(70,17){\circle*{3}} 
\put(40,17){\line(-1,0){10}} 
\put(40,17){\line(1,1){10}}
\put(40,17){\line(1,-1){10}} 
\put(60,17){\line(1,0){10}}
\put(60,17){\line(-1,1){10}} 
\put(60,17){\line(-1,-1){10}}
\put(50,37){\line(0,-1){10}} 
\put(50,-3){\line(0,1){10}} 
\end{picture}
\end{center}
\caption{Graphs $H_i$ (left) and $Sun_4$ (right)} 
\label{fig:H}
\end{figure}

Damaschke \cite{Damaschke} proved that the class of cographs is well-quasi-ordered by induced subgraphs. 
A cograph is a graph whose every induced subgraph with at least two vertices is either disconnected or 
the complement of a disconnected graph. The class of cographs is precisely the class of $P_4$-free graphs, 
i.e., graphs containing no $P_4$ as an induced subgraph. In \cite{Ding92}, Ding studied bi-cographs, i.e., 
the bipartite analog of cographs: these are bipartite graphs whose every induced subgraph with at least 
two vertices is either disconnected or the {\it bipartite complement} of a disconnected graph. Ding proved 
that the class of bi-cographs is also well-quasi-ordered by induced subgraphs. In terms of forbidden induced subgraphs
this is precisely the class of $(P_7,Sun_4,S_{1,2,3})$-free bipartite graphs \cite{Ding92} (see also \cite{GV97}), 
where $Sun_4$ is the graph represented in Figure~\ref{fig:H}(right) and $S_{1,2,3}$ is a tree with 3 leaves being 
of distance 1,2,3 from the only vertex of degree 3. 

Obviously, exclusion of an induced path is a necessary condition for a class of graphs defined 
by finitely many forbidden induced subgraphs to be well-quasi-ordered, since otherwise the class contains 
infinitely many cycles. It is also necessary for such classes to exclude the complement of an induced path, 
since the complements of cycles also form an antichain with respect to the induced subgraph relation. 
In the case of bipartite graphs, together with an induced path one also has to exclude the {\it bipartite complement} 
$\widetilde{P}_k$ of an induced path $P_k$. Excluding an induced path and the bipartite complement of an induced path is not, 
however, sufficient for a class of bipartite graphs to be well-quasi-ordered. In \cite{Ding92}, Ding found an infinite
antichain of $(P_8,\widetilde{P}_8)$-free bipartite graphs. On the other hand, he proved that $(P_6,\widetilde{P}_6)$-free 
bipartite graphs are well-quasi-ordered by induced subgraphs. Observe that the bipartite complement of a $P_7$ is
a $P_7$ again. The question whether the class of $P_7$-free bipartite graphs is well-quasi-ordered remained open
for about 20 years. In the present paper we answer this question negatively by exhibiting an antichain of 
$P_7$-free bipartite graphs. Moreover, we show that this antichain is also $Sun_4$-free. 
On the other hand, we show that $(P_7,Sun_1)$-bipartite graphs are well-quasi-ordered by the induced subgraph
relation, where $Sun_1$ is the graph obtained from $Sun_4$ by deleting 3 vertices of degree 1. We also
obtain two other positive results. First, we show that $(P_7,S_{1,2,3})$-free bipartite graphs are well-quasi-ordered 
by induced subgraphs, generalizing both the bi-cographs and $P_6$-free graphs. Second, we prove that 
$P_k$-free bipartite permutation graphs are well-quasi-ordered by induced subgraphs for any value of $k$.
The latter fact is in contrast with one more negative result of the present paper: by strengthening 
the Ding's idea, we show that $(P_8,\widetilde{P}_8)$-free bipartite graphs are not well-quasi-ordered 
even when restricted to {\it biconvex graphs}, a class generalizing bipartite permutation graphs. 
The relationship between the classes of graphs under consideration is represented in Figure~\ref{fig:inclusion}.

\begin{figure}
\begin{center}
\begin{picture}(400,200)

\put(150,170)
{
\oval(205,20)
\makebox(0,0)
{\scriptsize The class of   
$(P_8,\widetilde{P}_8)$-free bipartite graphs
}
}

\put(150,130){\oval(150,20)
\makebox(0,0)
{\scriptsize $(P_7,Sun_4)$-free bipartite graphs}
}

\put(310,130){\oval(150,20)
\makebox(0,0)
{\scriptsize $(P_8,\widetilde{P}_8)$-free biconvex graphs}
}



\put(265,60){\oval(150,20)
\makebox(0,0)
{\scriptsize $P_7$-free bipartite permutation graphs}}

\put(160,90){\oval(96,20)
\makebox(0,0)
{\scriptsize $(P_7,Sun_1)$-free bipartite}}

\put(60,90){\oval(96,20)
\makebox(0,0)
{\scriptsize $(P_7,S_{1,2,3})$-free bipartite}}

\put(325,90){\oval(150,20)
\makebox(0,0)
{\scriptsize $P_k$-free bipartite permutation graphs}}

\put(40,20){\oval(86,20)
\makebox(0,0)
{\scriptsize $(P_6,\widetilde{P}_6)$-free bipartite}}

\put(148,20){\oval(120,20)
\makebox(0,0)
{\scriptsize $(P_7,Sun_4,S_{1,2,3})$-free bipartite}}

\put(40,60){\oval(90,20)
\makebox(0,0)
{\scriptsize $P_6$-free bipartite graphs}}

\put(60,100){\line(0,1){60}}
\put(110,160){\line(0,-1){20}}
\put(110,120){\line(0,-1){90}}
\put(180,120){\line(0,-1){20}}
\put(210,120){\line(0,-1){50}}
\put(245,120){\line(0,-1){50}}
\put(245,160){\line(0,-1){20}}

\put(310,80){\line(0,-1){10}}
\put(40,30){\line(0,1){20}}
\put(40,70){\line(0,1){10}}
\put(95,30){\line(0,1){50}}
\qbezier[250](0,150)(200,150)(400,150)
\qbezier[250](0,40)(200,40)(400,40)
\put(0,110){\line(1,0){400}}

\put(-20,120){\it\small Not WQO}
\put(-20,95){\it\small WQO}
\put(293,153){\it\small Previously known results}
\put(350,142){\it\small New results}

\put(350,43){\it\small New results}
\put(293,32){\it\small Previously known results}

\end{picture}
\end{center}
\caption{Inclusion relationships between subclasses of bipartite graphs}
\label{fig:inclusion}
\end{figure}
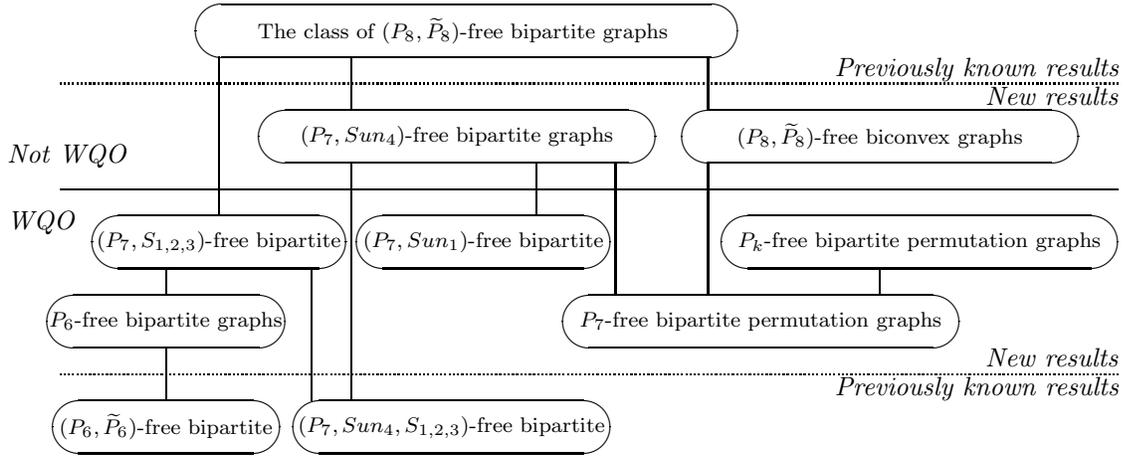
  
All graphs in this paper are undirected, without loops or multiple edges. 
The vertex set of a graph $G$ is denoted $V(G)$ and its edge set $E(G)$.
For a subset $U\subseteq V(G)$, by $G[U]$ we denote the subgraph of $G$
induced by $U$. The neighborhood of a vertex $v\in V(G)$ (i.e., the 
set of vertices of $G$ adjacent to $v$) is denoted $N_G(v)$. 
The degree of a vertex is the number of its neighbors. A graph is 1-regular 
if each of its vertices has degree 1. 

As usual, we denote by $P_n$, $C_n$ and $K_n$ the chordless path, the chordless cycle 
and the complete graph on $n$ vertices. Also, $2K_2$ is the disjoint union of two
copies of $K_2$. 

A graph is bipartite if the vertex set of the graph can be split into two parts
each of which is an independent set, i.e., a set of pairwise nonadjacent vertices. 
The bipartite complement of a bipartite graph $G=(V_1,V_2,E)$ with parts $V_1$ and $V_2$
and vertex set $E$ is a bipartite graph $\widetilde{G}=(V_1,V_2,V_1\times V_2-E)$.

We say that a graph $G$ is $H$-free if $G$ contains no copy of $H$ as an induced subgraph.
It is well known (and not difficult to see) that a $2K_2$-free bipartite graph possesses
the property that the vertices in each part of the graph can be linearly ordered under 
inclusion of their neighborhoods.   

\section{Not well-quasi-ordered classes of bipartite graphs}
In \cite{Ding92}, Ding proved that the class of $(P_8,\widetilde{P}_8)$-free bipartite graphs
is not well-quasi-ordered by the induced subgraph relation. In this section, we strengthen
this result in two ways. First, we show that $P_7$-free bipartite graphs are not wqo. 
Then we prove that $(P_8,\widetilde{P}_8)$-free {\it biconvex} graphs are not wqo.
To prove the results, in both cases we use the notion of a permutation,
i.e., a bijection of the set $[n]:=\{1,2,\ldots,n\}$ to itself. To represent a permutation 
$\pi :[n]\to [n]$, we use one of the following two ways: 
\begin{itemize}
\item one-line notation, which is the ordered sequence $(\pi(1),\pi(2),\ldots,\pi(n))$. 
\item a diagram (see Figure~\ref{fig:pi-12} for an example).  
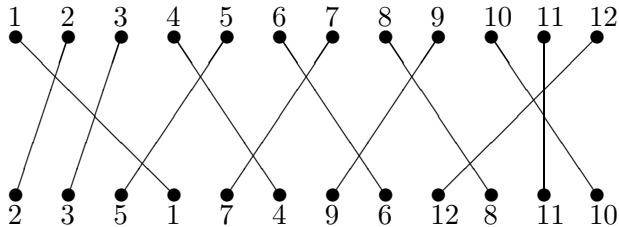
\begin{figure}[ht]
\begin{center}
\begin{picture}(160,70)
\put(0,0){\circle*{5}} 
\put(20,0){\circle*{5}}
\put(40,0){\circle*{5}} 
\put(60,0){\circle*{5}}
\put(80,0){\circle*{5}} 
\put(100,0){\circle*{5}}
\put(120,0){\circle*{5}} 
\put(140,0){\circle*{5}}
\put(160,0){\circle*{5}} 
\put(180,0){\circle*{5}}
\put(200,0){\circle*{5}} 
\put(220,0){\circle*{5}}

\put(0,60){\circle*{5}} 
\put(20,60){\circle*{5}}
\put(40,60){\circle*{5}} 
\put(60,60){\circle*{5}}
\put(80,60){\circle*{5}} 
\put(100,60){\circle*{5}}
\put(120,60){\circle*{5}} 
\put(140,60){\circle*{5}}
\put(160,60){\circle*{5}} 
\put(180,60){\circle*{5}}
\put(200,60){\circle*{5}} 
\put(220,60){\circle*{5}}

\put(0,0){\line(1,3){20}} 
\put(20,0){\line(1,3){20}}
\put(40,0){\line(2,3){40}} 
\put(60,0){\line(-1,1){60}}
\put(80,0){\line(2,3){40}} 
\put(100,0){\line(-2,3){40}}
\put(120,0){\line(2,3){40}} 
\put(140,0){\line(-2,3){40}}
\put(160,0){\line(1,1){60}} 
\put(180,0){\line(-2,3){40}}
\put(200,0){\line(0,1){60}}
\put(220,0){\line(-2,3){40}} 
\put(-3,-11){$2$} 
\put(17,-11){$3$}
\put(-3,64){$1$}
\put(17,64){$2$}
\put(37,-11){$5$} 
\put(57,-11){$1$}
\put(37,64){$3$}
\put(57,64){$4$}
\put(77,-11){$7$} 
\put(97,-11){$4$}
\put(77,64){$5$}
\put(97,64){$6$}

\put(117,-11){$9$} 
\put(137,-11){$6$}
\put(117,64){$7$}
\put(137,64){$8$}
\put(157,-11){$12$} 
\put(177,-11){$8$}
\put(197,-11){$11$} 
\put(217,-11){$10$}
\put(157,64){$9$}
\put(177,64){$10$}
\put(197,64){$11$} 
\put(217,64){$12$}
\end{picture}
\end{center}
\caption{The diagram representing the permutation $(2,3,5,1,7,4,9,6,12,8,11,10)$.} 
\label{fig:pi-12}
\end{figure}
\end{itemize}

The permutation graph $G_{\pi}$ of a permutation $\pi$ is the intersection graph of the digram representing $\pi$.
Figure~\ref{fig:pi} gives an example of a permutation and its permutation graph.

The composition $\mu\circ\rho$ of two permutations $\mu$ and $\rho$ is a permutation $\pi$  such that
$\pi(i)=\mu(\rho(i))$. The inverse of a permutation $\pi$ is a permutation  $\pi^{-1}$ such that 
$\pi^{-1}(\pi(i))=i$.

Let $\pi$ and $\rho$ be two permutations given in one-line notation. We say that $\pi$ is contained in $\rho$ 
if $\rho$ has a subsequence which is order-isomorphic to $\pi$. (Two sequences $(a_1,\ldots,a_n)$ and $(b_1,\ldots,b_n)$ are order-isomorphic if $a_i\le a_j$ if and only if $b_i\le b_j$.) It is not difficult to see from the diagram representations that if $G_{\pi}$ is not an induced subgraph of $G_{\rho}$, 
then $\pi$ is not contained in $\rho$. 

\subsection{The class of $(P_7,Sun_4)$-free bipartite graphs is not WQO}
\label{sec:notWQO}
We start by introducing a special class of bipartite graphs defined as follows:

\begin{dfn} For each permutation $\pi:=\pi_n$ on $[n]$, the graph $T:=T_{\pi}$
is a bipartite graph with parts $A\cup C$ and $B\cup D$, where:
\begin{enumerate} 
\item The vertex set of $T$ is the disjoint union of four independent vertex sets 
\begin{itemize}
\item $A := \{a_1, a_2, \ldots , a_n \}$,
\item $B := \{b_1, b_2, \ldots , b_n \}$, 
\item $C := \{c_1, c_2, \ldots , c_n \}$,
\item $D := \{d_1, d_2, \ldots , d_n \}$. 
\end{itemize}
\item $X(T) := T[A\cup B]$ is a 1-regular graph with $e_i := a_i b_{\pi(i)}$ being an edge for each $i\in [n]$.
\item $Y(T) := T[C\cup D]$ is a biclique (i.e., a complete bipartite graph).
\item Each of $Z'(T):= T[A\cup D]$ and $Z''(T) := T[B\cup C]$ is a $2K_2$-free 
bipartite graph defined as follows: for $i=1,2,\ldots,n$,
\begin{itemize}
\item $N_{Z'}(a_i)=\{d_1,\ldots,d_i\}$,
\item $N_{Z''}(b_i)=\{c_1,\ldots,c_i\}$.
\end{itemize}
\end{enumerate} 
Any graph of the form $T_{\pi}$ will be called a $T$-graph.
\end{dfn}

In order to derive the main result of this section, we will show that every $T$-graph is $(P_7,Sun_4)$-free
and that the set of $T$-graphs is not well-quasi-ordered by induced subgraphs. In fact, we will prove a 
slightly stronger result: we will show that every $T$-graph is $(2P_3,Sun_4)$-free,
where $2P_3$ is the graph obtained from $P_7$ by deleting the central vertex. 

\begin{lem}\label{easy} 
Any $T$-graph is $(2P_3,Sun_4)$-free. 
\end{lem}

\begin{proof} Suppose, for contradiction, that  $T:=T_{\pi}$ contains an induced $2P_3$. Then it is easy 
to check that each of the two $P_3$ must contain at least one vertex in each of $X(T)$ 
and $Y(T)$. Note that the vertices in $2P_3\cap Y(T)$ must all belong to the same part 
of the biclique $Y(T)$. We may assume without loss of generality that this part is $D$. 
It is clear that each $P_3$ has an edge from $A$ to $D$. But then $Z'(T)$ is not $2K_2$-free,
a contradiction showing that $T$ is $2P_3$-free. 

Now suppose, for contradiction, that $T$ contains an induced $Sun_4$.
Note that any two vertices in the same part of $Y(T)$ have nested 
neighborhoods. Therefore, no two vertices of degree 3 in the $Sun_4$
can belong to the same part of $Y(T)$. This implies that no two vertices 
of degree 3 in the $Sun_4$ can belong to the same part of $X(T)$.
Therefore, each of $A, B, C$ and $D$ must contain exactly one vertex of 
degree 3 in the $Sun_4$. Suppose that these vertices are $a, b, c$ and $d$, 
respectively. The leaf attached to $a$ in the $Sun_4$ cannot belong to $B$
(since otherwise $a$ has degree more than 1 in $X(T)$) and cannot belong 
to $D$ (since otherwise $Y(T)$ is not a biclique). This contradiction shows
that $T$ is $Sun_4$-free. 
\end{proof}

\medskip
Now we turn to showing that the set of $T$-graphs is not well-quasi-ordered by the induced subgraph relation.
To this end, for each even $n\ge 6$ we define a specific permutation $\pi_n^*$, as follows:
$$
\pi_n^* := (4,2,\ldots, 2j, 2j-5, \ldots, n-1,n-3) \ \  j=3,\ldots,n/2.
$$
For instance, $\pi_6^*=(4,2,6,1,5,3)$ and $\pi_8^*=(4,2,6,1,8,3,7,5)$.
For $n=10$, we use the diagram to represent $\pi_n^*$ (see Figure~\ref{fig:pi} (left)).
This diagram can also be interpreted as the subgraph $X(T)$ of $T_{\pi_{10}^*}$, which
can be seen by labeling the vertices in the upper part of the diagram by $a_1,\ldots,a_{10}$ consecutively from left to right
and the vertices in the lower part of the diagram by $b_1,\ldots,b_{10}$ consecutively from left to right.
The permutation graph $G_{\pi_{10}^*}$ of the permutation $\pi_{10}^*$ is represented in Figure~\ref{fig:pi} (right).

\begin{figure}[ht]
\begin{center}
\begin{picture}(220,90)
\put(0,0){\circle*{5}} 
\put(20,0){\circle*{5}}
\put(40,0){\circle*{5}} 
\put(60,0){\circle*{5}}
\put(80,0){\circle*{5}} 
\put(100,0){\circle*{5}}
\put(120,0){\circle*{5}} 
\put(140,0){\circle*{5}}
\put(160,0){\circle*{5}} 
\put(180,0){\circle*{5}}

\put(0,60){\circle*{5}} 
\put(20,60){\circle*{5}}
\put(40,60){\circle*{5}} 
\put(60,60){\circle*{5}}
\put(80,60){\circle*{5}} 
\put(100,60){\circle*{5}}
\put(120,60){\circle*{5}} 
\put(140,60){\circle*{5}}
\put(160,60){\circle*{5}} 
\put(180,60){\circle*{5}}

\put(0,0){\line(1,1){60}} 
\put(20,0){\line(0,1){60}}
\put(40,0){\line(1,1){60}} 
\put(60,0){\line(-1,1){60}}
\put(80,0){\line(1,1){60}} 
\put(100,0){\line(-1,1){60}}
\put(120,0){\line(1,1){60}} 
\put(140,0){\line(-1,1){60}}
\put(160,0){\line(0,1){60}} 
\put(180,0){\line(-1,1){60}} 
\put(-3,-11){$4$} 
\put(17,-11){$2$}
\put(-3,64){$1$}
\put(17,64){$2$}
\put(37,-11){$6$} 
\put(57,-11){$1$}
\put(37,64){$3$}
\put(57,64){$4$}
\put(77,-11){$8$} 
\put(97,-11){$3$}
\put(77,64){$5$}
\put(97,64){$6$}

\put(117,-11){$10$} 
\put(137,-11){$5$}
\put(117,64){$7$}
\put(137,64){$8$}
\put(157,-11){$9$} 
\put(177,-11){$7$}
\put(157,64){$9$}
\put(177,64){$10$}
\end{picture}
\begin{picture}(130,70)
\setlength{\unitlength}{0.6mm}
\put(10,20){\circle*{3}} 
\put(20,10){\circle*{3}}
\put(40,10){\circle*{3}} 
\put(60,10){\circle*{3}}
\put(80,10){\circle*{3}} 
\put(90,20){\circle*{3}} 
\put(20,30){\circle*{3}}
\put(40,30){\circle*{3}} 
\put(60,30){\circle*{3}}
\put(80,30){\circle*{3}} 

\put(20,30){\line(1,0){20}} 
\put(20,30){\line(0,-1){20}}
\put(20,30){\line(-1,-1){10}} 
\put(20,10){\line(-1,1){10}}
\put(40,10){\line(1,0){20}} 
\put(40,10){\line(-1,0){20}}
\put(40,10){\line(0,1){20}} 
\put(60,30){\line(-1,0){20}}
\put(60,30){\line(1,0){20}}
\put(60,30){\line(0,-1){20}} 
\put(80,10){\line(-1,0){20}} 
\put(80,10){\line(1,1){10}}
\put(80,10){\line(0,1){20}} 
\put(80,30){\line(1,-1){10}}

\put(19,3){$4$} 
\put(19,32){$1$}
\put(5,18){$2$}
\put(39,3){$3$}
\put(39,32){$6$} 
\put(59,3){$8$}
\put(59,32){$5$}
\put(79,3){$7$}
\put(79,32){$10$} 
\put(92,18){$9$}

\end{picture}
\end{center}
\caption{The permutation $\pi_{10}^*$ (left) and the permutation graph $G_{\pi_{10}^*}$ (right)} 
\label{fig:pi}
\end{figure}
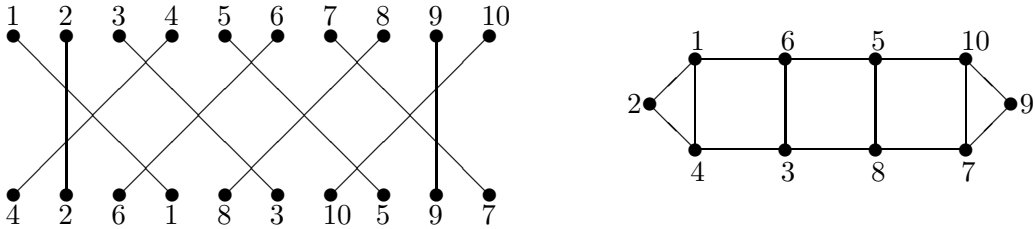

The important fact about the permutations $\pi_n^*$ is that
\begin{claim}\label{perm}
The sequence $\pi_6^*,\pi_8^*,\pi_{10}^*\ldots$ is an antichain of permutations with respect 
to the containment relation.  
\end{claim}

This claim follows directly from the fact that no graph $G_{\pi_n^*}$ is an induced subgraph of 
$G_{\pi_m^*}$ with $n\ne m$, which can be easily seen. We now use Claim~\ref{perm} in order to prove the following result.

\begin{lem}\label{main}
The sequence $T_{\pi_6^*},T_{\pi_8^*},T_{\pi_{10}^*},\ldots$ is an antichain with respect to 
the induced subgraph relation.  
\end{lem}

\begin{proof}
Suppose by contradiction that there is a graph $H:=T_{\pi_m^*}$ which is an induced subgraph of
a graph $G:=T_{\pi_n^*}$ for some even $6\le m<n$. We fix an arbitrary embedding of $H$ into 
$G$, i.e., we assume that $V(H)\subset V(G)$. We will denote the vertex subsets $A,B,C,D$ of 
the graph $H$ by $A(H), B(H), C(H), D(H)$ and of the graph $G$ by $A(G), B(G), C(G), D(G)$. 
Since both graphs are connected bipartite and $\pi_n = \pi_n^{-1}$, then in both graphs the role of the parts $A\cup C$ and $B\cup D$ 
is symmetric, so we may assume that

\begin{claim}\label{simple} 
$A(H)\cup C(H) \subseteq A(G)\cup C(G)$ and $B(H)\cup D(H) \subseteq B(G)\cup D(G)$. 
\end{claim}
Keeping Claim~\ref{simple} in mind, we derive a series of conclusions. First,   
we show that

\begin{claim}\label{ineq} 
$|A(H)\cap C(G)| \leq 1$, $|B(H)\cap D(G)| \leq 1$, 
$|C(H)\cap A(G)| \leq 1$, $|D(H)\cap B(G)| \leq 1$.
\end{claim}

\begin{proof}
Suppose $|A(H)\cap C(G)| \geq 2$, and pick two distinct vertices $a_i, a_j\in A(H)$ 
that belong to $C(G)$. Let $\pi := \pi_m^*$.
Since $Y(G)$ is a biclique, both $b_{\pi(i)}$ and $b_{\pi(j)}$ must lie in $B(G)$,
which contradicts the $2K_2$-freeness of $Z''(G)$. 
Thus $|A(H)\cap C(G)| \leq 1$. The second inequality follows by symmetry.

Suppose $c_i,c_j\in C(H)\cap A(G)$ ($i<j$). Since $b_j\in B(H)$ is
adjacent to both $c_i$ and $c_j$, we have $b_j\in D(G)$. Since
$|B(H)\cap D(G)|\le 1$, $b_i\in B(G)$.  Then $a_{\pi^{-1}(i)}$ is
adjacent to $b_i$ but $b_i$ has only one neighbor $c_i$ in $A(G)$, and
therefore $a_{\pi^{-1}(i)}\in C(G)$. Then $a_{\pi^{-1}(i)}\in C(G)$ is adjacent
to $b_j\in D(G)$, a contradiction.  This proves that $|C(H)\cap A(G)|\le
1$. The fourth inequality follows by symmetry.

\end{proof}

\medskip
Now we prove that

\begin{claim}\label{claim1} 
$|X(H)\cap Y(G)| \leq 1$ and $|Y(H)\cap X(G)| \leq 1$. 
\end{claim}

\begin{proof} 
By Claim~\ref{ineq} and the definition of $Y(G)$, if the intersection $X(H)\cap Y(G)$ contains two vertices, 
then these vertices must be adjacent. 
Let $\pi := \pi_m^*$ and suppose an edge $a_i b_{\pi(i)}$ of $X(H)$ belongs to $Y(G)$. 
By Claim~\ref{ineq}, $|D(H)\cap B(G)| \leq 1$, which means that $a_i$ is adjacent 
to all but at most one vertex of $D(H)$. According to the definition of $H$, we 
conclude that $i\in\{m-1,m\}$. Similarly, $b_{\pi(i)}$ is adjacent to all but at most one vertex of $C(H)$,
implying that $\pi(i)\in\{m-1,m\}$. Together $i\in\{m-1,m\}$ and $\pi(i)\in\{m-1,m\}$ imply $i=\pi(i)=m-1$.
From this and Claim~\ref{ineq} we conclude that both $a_m\in A(H)$ and $b_m\in B(H)$ belong to $X(G)$. Also, since 
\begin{itemize}
\item $a_{m-1}\in A(H)$ belongs to $Y(G)$,
\item $a_{m-1}$ is not adjacent to $d_m\in D(H)$ in $H$ and 
\item $Y(G)$ is a biclique, 
\end{itemize}
we conclude that $d_m\in D(H)$ belongs to $X(G)$. 
Similarly, $c_m\in C(H)$ belongs to $X(G)$. This contradicts the 1-regularity of $X(G)$, since $c_m\in A(G)$ is adjacent to both
$b_m$ and $d_m$ in $B(G)$. Thus $|X(H)\cap Y(G)| \leq 1$.

The prove the second inequality, suppose there is an edge $cd$ of $Y(H)$ belonging to $X(G)$. 
By definition, vertex $c\in C(H)$ must have a neighbor in $B(H)$, and due to 1-regularity of 
$X(G)$ this neighbor must belong to $D(G)$. Similarly, $d$ must have a neighbor in $A(H)\cap C(G)$. 
But this contradicts $|X(H)\cap Y(G)| \leq 1$. 
\end{proof}

\medskip
Next, we show that

\begin{claim}\label{claim2} 
$X(H)\cap Y(G) = Y(H)\cap X(G) = \emptyset$. 
\end{claim}

\begin{proof}
Assume first that $X(H)\cap Y(G)$ is not empty, and suppose without loss of 
generality that a vertex $a_i$ of $A(H)$ belongs to $Y(G)$. Then by Claim~\ref{claim1}
all vertices of $B(H)$ belong to $X(G)$. By Claim~\ref{ineq}, $|D(H)\cap B(G)| \leq 1$, 
which means that $a_i$ is adjacent to all but at most one vertex of $D(H)$. According 
to the definition of $H$, we conclude that $i=m-1$ or $i=m$. In either case, vertex 
$b_m$ is not adjacent to $a_i$, and the neighborhood of $b_m$ in the graph $Z''(H)$
is strictly greater than the neighborhood of $b_{\pi(i)}$.  

Suppose $i=m$. By Claim~\ref{ineq}, at least one of 
$c_{m-1},c_m\in C(H)$ belongs to $C(G)$, say $c_m\in C(G)$. But then $a_m,b_{m-3},b_m,c_m$
induce a $2K_2$, contradicting the $2K_2$-freeness of $Z''(G)$.

Suppose now that $i=m-1$. By definition, the vertex $a_{m-1}$ of $A(H)$ has a non-neighbor in $D(H)$.
Therefore, the set $D(H)$ must have a vertex in $X(G)$. This implies by Claim~\ref{claim1} that $C(H)\subset C(G)$,
and hence the vertices $a_{m-1},b_{m-1},c_m,b_m$ induce a $2K_2$, contradicting
the $2K_2$-freeness of $Z''(G)$. This completes the proof of the fact that 
$X(H)\cap Y(G) = \emptyset$.

\medskip
Now assume that $Y(H)\cap X(G) \not = \emptyset$ and suppose without loss of generality 
that a vertex $d_i$ of $D(H)$ belongs to $X(G)$. Since $X(H)\cap Y(G) = \emptyset$, the vertex $a_i\in A(H)$ lies in $A(G)$ and has two neighbors $b_{\pi(i)}$ and $d_i$ in
$B(G)$, a contradiction. Therefore, $Y(H)\cap X(G) = \emptyset$.   
\end{proof}

\medskip
Claims~\ref{claim2} and Claim~\ref{simple} together imply the following conclusion.

\begin{claim}\label{eq:inclusion}
$A(H) \subseteq A(G),\ B(H) \subseteq B(G),\ C(H) \subseteq C(G)$ and $D(H) \subseteq D(G)$.
\end{claim}

Assuming that $H$ is an induced subgraph of $G$, we must conclude that the ordering of vertices of $A(H)$
respects the ordering of vertices of $A(G)$, and similarly, the ordering of vertices of $B(H)$
respects the ordering of vertices of $B(G)$. But then we must conclude that $\pi_m^*$ is contained in
$\pi_n^*$ which is a contradiction to Claim~\ref{perm}. This contradiction completes the proof of the lemma.  
\end{proof}

\medskip
Lemmas~\ref{easy} and \ref{main} imply the main result of this section.

\begin{thm} 
The class of $(P_7, Sun_4)$-free bipartite graphs is not well-quasi-ordered by the induced subgraph relation. 
\end{thm}

\subsection{The class of $(P_8, \widetilde{P}_8)$-free biconvex graphs is not WQO}
\label{sec:biconnotWQO}

A bipartite graph is \emph{biconvex} if the vertices of the graph can be 
linearly ordered so that the neighborhood of each vertex forms an interval, 
i.e., the neighborhood consists of consecutive vertices in the order. \cite{Tucker}
Strengthening the result from \cite{Ding92}, we show in this section that 
the class of $(P_8, \widetilde{P}_8)$-free {\it biconvex} graphs is not wqo by 
the induced subgraph relation. We start by introducing two special types 
of permutations.

\begin{dfn}
A permutation $\pi_n$ is {\it convex} if for any $1\le i \le n$ 
the set $\pi_n^{-1}(\{i,i+1,\ldots,n-1,n\})$ forms an interval, i.e., the
elements of the set occupy consecutive positions in the permutation.
\end{dfn} 

For instance, the permutation $\rho=(1,2,3,5,7,9,10,8,6,4)$ is convex. Indeed, 
the elements of the set $\{5,6,7,8,9,10\}$ occupy positions $4,5,6,7,8,9$, 
the elements of the set $\{6,7,8,9,10\}$ occupy positions $5,6,7,8,9$, and the 
same is true for any other set of the form $\{i,i+1,\ldots,n-1,n\}$. 
The permutation $\mu=(2,3,5,7,10,9,8,6,4,1)$ is another example of a convex 
permutation.

\begin{dfn}
A permutation $\pi$ is {\it biconvex} if there are two convex permutations
$\mu$ and $\rho$ such that $\pi=\mu\circ \rho^{-1} $.
\end{dfn} 

To give an example, consider the following permutation:
$\pi=(2,3,5,1,7,4,10,6,9,8)$. It is not difficult to verify that 
$\pi=\mu\circ \rho^{-1}$, where $\mu$ and $\rho$ are the two convex
permutations given above. For instance, $\pi(1)=\mu(\rho^{-1}(1))=2$, 
$\pi(2)=\mu(\rho^{-1}(2))=3$, $\pi(3)=\mu(\rho^{-1}(3))=5$, etc.
Therefore, $\pi$ is a biconvex permutation. 

\medskip
By $\pi[\mu,\rho]$ we shall denote a biconvex permutation $\pi$ 
given together with a pair of convex permutations  $\mu$ and $\rho$ such
that $\pi=\mu\circ \rho^{-1}$. 
Now we introduce a special class of bipartite graphs defined as follows:

\begin{dfn} For a biconvex permutation $\pi:=\pi_n[\mu,\rho]$, 
the graph $S:=S_{\pi}$ is a bipartite graph with parts $A\cup C$ and $B$, where:
\begin{enumerate} 
\item $V(S)$ is the disjoint union of three independent vertex sets 
\begin{itemize}
\item $A := \{a_1, a_2, \ldots , a_n \}$,
\item $B := \{b_1, b_2, \ldots , b_n \}$, 
\item $C := \{c_1, c_2, \ldots , c_n \}$,
\end{itemize}

\item Each of $X(S):= S[A\cup B]$ and $Y(S) := S[B\cup C]$ is a $2K_2$-free 
bipartite graph defined as follows: for $i=1,2,\ldots,n$,
\begin{itemize}
\item $N_{X}(b_i)=\{a_1,\ldots,a_{\rho(i)}\}$,
\item $N_{Y}(b_i)=\{c_1,\ldots,c_{\mu(i)}\}$.
\end{itemize}
\end{enumerate} 
Any graph of the form $S_{\pi}$ will be called an $S$-graph.
\end{dfn}

\begin{rem}
$N_G(a_i)=\{b_{\rho^{-1}(i)},\ldots,b_{\rho^{-1}(n)}\}$
and
$N_G(c_i)=\{b_{\mu^{-1}(i)},\ldots,b_{\mu^{-1}(n)}\}$.
\end{rem}

\begin{claim}\label{easyclaim} 
Any $S$-graph is a $(P_8, \widetilde{P}_8)$-free biconvex graph. 
\end{claim}

\begin{proof}
Let $S:=S_{\pi}$ be an $S$-graph associated with a biconvex permutation $\pi:=\pi_n$ 
such that $\pi=\mu\circ \rho^{-1}$, where $\mu$ and $\rho$ are two convex permutations.
The $(P_8, \widetilde{P}_8)$-freeness of $S$ follows from the $2K_2$-freeness of $X(T)$ and $Y(T)$.
Now let us prove that $S$ is biconvex. To this end, we need to show that 
the vertices in each part of the graph can be linearly ordered so that the neighborhood 
of any vertex in the opposite part forms an interval. To achieve this goal we keep the 
natural order of the vertices in the $B$-part, i.e., $B=(b_1,\ldots, b_n)$. 
The vertices of the $A\cup C$-part are ordered under inclusion of their neighborhoods,
increasingly for the $A$-vertices and decreasingly for the $C$-vertices, i.e., 
the vertices with the largest neighborhood in $A$ and $C$ are in the middle of the order.
Now let us show that the defined order is biconvex.

Let $b$ be any vertex from $B$. If $b$ is adjacent to any vertex $a$ from $A$, then 
$b$ is adjacent to any vertex from $A$ with larger neighborhood than $N(a)$, i.e.,
$b$ is adjacent to any vertex of $A$ following $a$. Similarly, 
if $b$ is adjacent to any vertex $c$ from $C$, then 
$b$ is adjacent to any vertex from $C$ with larger neighborhood than $N(c)$,
i.e., $b$ is adjacent to any vertex of $C$ preceding $c$. Therefore,
$N(b)$ is an interval.

Now let $a_i$ be a vertex from $A$. Let $I$ be the interval (i.e., the set of positions) 
of length $n-i+1$ containing the elements $\{i,\ldots,n\}$ of the permutation $\rho$.  
Then $N(a_i)=\{b_j\ :\ j\in I\}$, i.e., $N(a_i)$ is an interval. Similarly, 
if $c_i$ is a vertex from $C$ and $I$ is the interval of length $n-i+1$ containing 
the elements $\{i,\ldots,n\}$ of the permutation $\mu$, then $N(c_i)=\{b_j\ :\ j\in I\}$, 
i.e., $N(c_i)$ is an interval
\end{proof}

\medskip
Now we define a specific permutation $\pi_n^*$ in the following way: 
for each even $n\ge 8$, 
$$
\pi_n^* := (2,3,5,1\ldots, 2j+3, 2j, \ldots, n,n-4,n-1,n-2) \ \  j=2,\ldots,n/2-4.
$$
For instance, $\pi_8^*=(2,3,5,1,8,4,7,6)$ and $\pi_{10}^*=(2,3,5,1,7,4,10,6,9,8)$. 
The permutation $\pi_{12}^*$ is represented in Figure~\ref{fig:pi-12}.

Let us show that $\pi_n^*$ is a
biconvex permutation. To this end, we define two convex permutations $\rho_n^*$ and 
$\mu_n^*$ in the following way:
$$
\rho_n^* := (1,2,3,5\ldots,\mbox{odd numbers},\ldots,n-3,n-1,n,n-2,\ldots,\mbox{even numbers},\ldots,6,4).
$$
$$
\mu_n^* := (2,3,5\ldots,\mbox{odd numbers},\ldots,n-3,n,n-1,n-2,n-4,\ldots,\mbox{even numbers},\ldots,6,4,1).
$$
It is not difficult to verify that for $n=10$ the permutations $\pi_n^*$, $\rho_n^*$ and 
$\mu_n^*$ coincide with the permutations $\pi$, $\rho$ and $\mu$ defined in the beginning of the section.

\begin{claim} 
$\pi_n^*=\mu_n^*\circ \rho_n^{*-1} $.
\end{claim}
\begin{proof}
For small and large values of $i$, one can verify by direct inspection that $\pi_n^*(i)=\mu_n^*(\rho_n^{*-1}(i))$.
Now let $4<i<n-3$. If $i$ is odd then $\pi_n^*(i)=\mu_n^*(\rho_n^{*-1}(i))=i+2$, and if $i$ is even then
$\pi_n^*(i)=\mu_n^*(\rho_n^{*-1}(i))=i-2$. 
\end{proof}

\begin{lem}\label{biconmain}
The sequence $S_{\pi_8^*},S_{\pi_{10}^*},S_{\pi_{12}^*},\ldots$ is an antichain with respect to 
the induced subgraph relation.  
\end{lem}

\begin{proof}
Suppose by contradiction that there is a graph $H:=S_{\pi_m^*}$ which is an induced subgraph 
of a graph $G:=S_{\pi_n^*}$ for some even $8\le m<n$. We fix an arbitrary embedding of $H$ into 
$G$, i.e., we assume that $V(H)\subset V(G)$. Since both graphs are connected bipartite, 
we may assume that exactly one of the following two possibilities holds:

\medskip

\begin{enumerate}
\item $A(H)\cup C(H) \subseteq A(G)\cup C(G)$ and $B(H)\subseteq B(G)$. 
\item $A(H)\cup C(H) \subseteq B(G)$ and $B(H)\subseteq A(G)\cup C(G)$
\end{enumerate}

We claim that the first possibility holds.

\begin{claim}
$A(H)\cup C(H) \subseteq A(G)\cup C(G)$ and $B(H) \subseteq B(G)$. 
\end{claim}

\begin{proof}
Note that, by definition, $A(G)\cup C(G)$ can be partitioned into 
two chains with respect to the neighborhood inclusion. On the other hand, 
the set $B(H)$ does not have this property, since $b_{n/2}, b_{n/2+1}, b_{n/2+2}$ 
is an antichain of length 3 with respect to the same relation. Indeed, 
$\rho^*(n/2)=n-3$, $\rho^*(n/2+1)=n-1$, $\rho^*(n/2+2)=n$ and
$\mu^*(n/2)=n$, $\mu^*(n/2+1)=n-1$ and $\mu^*(n/2+2)=n-2$.
This proves the claim.
\end{proof}

\medskip

We make the following helpful remark:

\begin{itemize}
\item If two vertices of $B(H)$ are incomparable with respect to the neighborhood inclusion in $B(H)$, 
then these two vertices must also be incomparable with respect to the neighborhood inclusion in $B(G)$.
\end{itemize}

Let $B'(H)$ be the incomparability graph for the relation of neighborhood inclusion on the vertex set $B(H)$. 
In other words, two vertices of $B(H)$ are adjacent in $B'(H)$ precisely when they are incomparable with respect 
to the neighborhood inclusion. We define $B'(G)$ similarly.

\medskip
Clearly, by the above remark, $B'(H)$ must be a subgraph of $B'(G)$. 
But for any even $n \ge 8$, the graph $B'(S_{\pi_n^*})$ is simply the permutation graph $G_{\pi_n^*}$ of $\pi_n^*$
and this graph is represented in Figure~\ref{fig:pn}. 

\begin{figure}[ht]
\begin{center}
\begin{picture}(150,50)
\put(-10,25){\circle*{4}} 
\put(20,25){\circle*{4}}
\put(50,25){\circle*{4}} 
\put(70,25){\circle{1}}
\put(80,25){\circle{1}} 
\put(90,25){\circle{1}}
\put(110,25){\circle*{4}} 
\put(140,25){\circle*{4}}
\put(-30,45){\circle*{4}} 
\put(-30,5){\circle*{4}}
\put(160,45){\circle*{4}} 
\put(160,5){\circle*{4}}
\put(-8,25){\line(1,0){26}} 
\put(22,25){\line(1,0){26}}
\put(52,25){\line(1,0){8}} 
\put(100,25){\line(1,0){8}}
\put(112,25){\line(1,0){26}} 
\put(-10,25){\line(-1,1){20}}
\put(-10,25){\line(-1,-1){20}} 
\put(140,25){\line(1,1){20}}
\put(140,25){\line(1,-1){20}}
\put(160,45){\line(0,-1){40}} 
\put(0,30){1} 
\put(38,30){2}
\put(115,30){$n-5$}
\end{picture}
\end{center}
\caption{The graph $B'(S_{\pi_n^*})=G_{\pi_n^*}$} 
\label{fig:pn}
\end{figure}
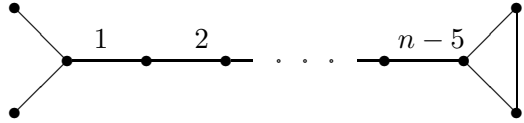

It is not difficult to see that the sequence of graphs $G_{\pi_n^*}$, $n\ge 8$, forms
an antichain with respect to the (induced) subgraph relation.  
Therefore, $B'(H)$ is not a subgraph of $B'(G)$. As a result, $H$ is not an induced subgraph of $G$. 
This contradiction completes the proof of Lemma~\ref{biconmain}.
\end{proof}

\medskip

Lemma~\ref{biconmain} and Claim~\ref{easyclaim} together imply the main result of this section:

\begin{thm} 
The class of $(P_8, \widetilde{P}_8)$-free biconvex graphs is not well-quasi-ordered by the induced subgraph relation. 
\end{thm}

\section{Well-quasi-ordered classes of bipartite graphs}
In this section, we turn to positive results, i.e., to classes of bipartite graphs which are well-quasi-ordered
by the induced subgraph relation. 

\subsection{The class of $(P_7, S_{1,2,3})$-free bipartite graphs}

In \cite{Ding92}, Ding showed that $(P_7, S_{1,2,3},Sun_4)$-free bipartite graphs and $(P_6,\widetilde{P}_6)$-free
bipartite graphs are well-quasi-ordered by the induced subgraph relation. Now we extend both results
to the larger class of $(P_7, S_{1,2,3})$-free bipartite graphs. To this end, let us introduce the following
notation.

Given a set of bipartite graphs ${\cal F}$, we denote by $[{\cal F}]$ the set of bipartite graphs constructed from graphs in $\cal F$
by means of the following three binary operations defined for any two disjoint bipartite graphs 
$G_1=(X_1,Y_1,E_1)$ and $G_2=(X_2,Y_2,E_2)$:
\begin{itemize}
\item the disjoint union is the operation that creates out of $G_1$ and $G_2$ the bipartite graph
$G=(X_1\cup X_2,Y_1\cup Y_2,E_1\cup E_2)$,
\item the join is the operation that creates out of $G_1$ and $G_2$ the bipartite graph which is 
the bipartite complement of the disjoint union of $\widetilde{G}_1$ and 
$\widetilde{G}_2$,
\item the skew join is the operation that creates out of $G_1$ and $G_2$ the bipartite graph 
$G=(X_1\cup X_2,Y_1\cup Y_2,E_1\cup E_2\cup\{xy\ :\ x\in X_1,y\in Y_2\})$.
\end{itemize}

The importance of these operations is due to the following theorem. 

\begin{thm}\label{[H]}
If $\cal F$ is a set of bipartite graphs well-quasi-ordered by the induced subgraph relation, then so is $[{\cal F}]$.
\end{thm}

For the proof of this theorem, we refer the reader to Theorems 4.1 and 4.4 from \cite{Ding92},
where the author used this result (without formulating it implicitly) in his proof that 
$(P_7, S_{1,2,3}$, $Sun_4)$-free bipartite graphs and $(P_6,\widetilde{P}_6)$-free bipartite 
graphs are will-quasi-ordered by the induced subgraph relation. 
Now we combine Theorem~\ref{[H]} with a result from \cite{FGV00} that can be formulated as follows.

\begin{thm}\label{weak}
The class of $(P_7, S_{1,2,3})$-free bipartite graphs is precisely $[\{K_1\}]$.
\end{thm}

Together, Theorem~\ref{[H]} and Theorem~\ref{weak} imply the following conclusion.

\begin{thm}
The class of $(P_7, S_{1,2,3})$-free bipartite graphs is well-quasi-ordered by the induced subgraph relation. 
\end{thm}

\subsection{The class of $(P_7, Sun_1)$-free bipartite graphs}
The graph $Sun_1$ is obtained from $Sun_4$ (Figure~\ref{fig:H}) by deleting three vertices of degree 1.
Therefore, the class of $(P_7, Sun_1)$-free bipartite graphs is a proper subclass of $(P_7, Sun_4)$-free 
bipartite graphs. In contrast to the result of Section~\ref{sec:notWQO}, below we prove that $(P_7, Sun_1)$-free 
bipartite graphs are well-quasi-ordered by the induced subgraph relation. According to Theorem~\ref{[H]},
it suffices to show that the set of {\it connected} $(P_7, Sun_1)$-free bipartite graphs is well-quasi-ordered
by this relation. The following lemma shows that the structure of connected graphs in this class containing a $C_4$
is rather simple.

\begin{lem}\label{lem:reduction}
Every connected $(P_7, Sun_1)$-free bipartite graph containing a $C_4$ is complete bipartite.
\end{lem}

\begin{proof}
Let $H$ be a $(P_7, Sun_1)$-free bipartite graph containing a $C_4$.
Denote by $H'$ any maximal complete bipartite subgraph of $H$ containing the $C_4$. 
If $H'\ne H$, there must exist a vertex $v$ outside $H'$ that has a neighbor in $H'$. 
If $v$ is a adjacent to every vertex of $H'$ in the opposite part, then $H'$ is not maximal, 
and if $v$ has a non-neighbor in the opposite part of $H'$, the reader can easily find an induced $Sun_1$. 
The contradiction in both cases shows that $H'=H$, i.e., $H$ is a complete bipartite graph. 
\end{proof}

\medskip
It is not difficult to see that there is no infinite antichain of complete bipartite graphs, which follows, for instance,
from the fact that every complete bipartite graph is $P_4$-free and the class of $P_4$-free (not necessarily bipartite) 
graphs is well-quasi-ordered. This observation together with Lemma~\ref{lem:reduction} reduces the problem 
from $(P_7, Sun_1)$-free bipartite graphs to $(P_7,C_4)$-free bipartite graphs. 
The proof that the class of $(P_7,C_4)$-free bipartite graphs is well-quasi-ordered
is based on the following lemma.

\begin{lem}\label{key} 
No $(P_7, C_4)$-free bipartite graph contains $P_9$ as a subgraph (not necessarily induced).  
\end{lem}

\begin{proof}
Let $G$ be a $(P_7, C_4)$-free bipartite graph. To prove the lemma, we first derive the following 
helpful observation.

\begin{claim}\label{p7} If $P := (a_1, a_2, \ldots, a_7)$ is a copy of $P_7$ contained in $G$ as a subgraph, 
then $P$ has exactly one chord in $G$, either $a_1 a_6$ or $a_2 a_7$. 
\end{claim}

\begin{proof}
Since $G$ is $P_7$-free, $P$ must contain a chord, and since $G$ is bipartite, any chord 
of $P$ connects an even-indexed vertex to an odd-indexed one. Among 6 possible
chords of $P$ only $a_1a_6$ and $a_2a_7$ do not produce a $C_4$, and these two chords 
cannot be present in the graph simultaneously, since otherwise the vertices $a_1, a_2, a_7, a_6$ induce a $C_4$.
Therefore, $P$ must contain exactly one of $a_1 a_6$ or $a_2a_7$ as a chord.
\end{proof}

\medskip
Suppose now that $Q:= (b_1, b_2, \ldots, b_9)$ is a copy of $P_9$ contained as a subgraph in $G$,
and for $1\leq i\leq 3$, let $Q_i := (b_i b_{i+1} \ldots b_{i+6})$. If $b_1 b_6$ is a chord of $Q$, 
then Claim~\ref{p7} applied to each of $Q_1, Q_2$ and $Q_3$
implies that $Q$ contains exactly two chords, namely $b_1 b_6$ and $b_3 b_8$. But then 
the vertices $b_1, b_6, b_5, b_4, b_3, b_8, b_9$ induce a $P_7$, a contradiction. 

The case when $b_1 b_6$ is not a chord of $Q$ is symmetric and also leads (with the help of Claim~\ref{p7})
to an induced $P_7$ in $G$. The contradiction in both cases shows that $G$ does not contain $P_9$
as a subgraph. 
\end{proof}

\medskip
Now we combine Lemma~\ref{key} with the following result by Ding \cite{Ding92}.  

\begin{thm}[Ding \cite{Ding92}]\label{ding} 
For any fixed $k\ge 1$, the class of graphs containing no $P_k$ as a (not necessarily induced) subgraph
is well-quasi-ordered by the induced subgraph relation.  
\end{thm}

Together Lemma~\ref{key} and Theorem~\ref{ding} imply the main conclusion of this section. 
 
\begin{thm}\label{ours} 
The class of $(P_7, C_4)$-free bipartite graphs is well-quasi-ordered by the induced subgraph relation. 
\end{thm}

\subsection{The class of $P_k$-free bipartite permutation graphs}
The class of bipartite permutation graphs is the intersection of bipartite graphs and permutation graphs.
This class is a subclass of biconvex graphs (see e.g. \cite{book}). In contrast to the result of 
Section~\ref{sec:biconnotWQO} we show that $P_k$-free bipartite permutation graphs are well-quasi-ordered 
by the induced subgraph relation for any fixed value of $k$. In general, bipartite permutation graphs are 
not well-quasi-ordered by this relation, since they contain the antichain of graphs of the form $H_i$ (Figure~\ref{fig:H}). 
Our proof is based on a number of known results. 

\medskip
Denote by $H_{n,m}$ the graph with $nm$ vertices which can be partitioned into $n$ independent sets
$V_1=\{v_{1,1},\ldots,v_{1,m}\}$, $\ldots$, $V_n=\{v_{n,1},\ldots,v_{n,m}\}$ so that for each $i=1,\ldots,n-1$
and for each $j=1,\ldots,m$, vertex $v_{i,j}$ is adjacent to vertices $v_{i+1,1}, v_{i+1,2},\ldots,v_{i+1,j}$
and there are no other edges in the graph. In other words, every two consecutive independent sets induce in $H_{n,m}$
a universal chain graph. An example of the graph $H_{n,n}$ with $n=5$ is given in
Figure~\ref{fig:H55}.

\begin{figure}[ht]
\begin{center}
\begin{picture}(300,200)
\put(50,0){\circle*{5}}
\put(100,0){\circle*{5}}
\put(150,0){\circle*{5}}
\put(200,0){\circle*{5}}
\put(250,0){\circle*{5}}
\put(500,50){\circle*{5}}
\put(100,50){\circle*{5}}
\put(150,50){\circle*{5}}
\put(200,50){\circle*{5}}
\put(250,50){\circle*{5}}
\put(50,100){\circle*{5}}
\put(50,50){\circle*{5}}
\put(100,100){\circle*{5}}
\put(150,100){\circle*{5}}
\put(200,100){\circle*{5}}
\put(250,100){\circle*{5}}
\put(50,150){\circle*{5}}
\put(100,150){\circle*{5}}
\put(150,150){\circle*{5}}
\put(200,150){\circle*{5}}
\put(250,150){\circle*{5}}
\put(50,200){\circle*{5}}
\put(100,200){\circle*{5}}
\put(150,200){\circle*{5}}
\put(200,200){\circle*{5}}
\put(150,200){\circle*{5}}
\put(250,200){\circle*{5}}
\put(50,0){\line(0,1){50}}
\put(50,0){\line(1,1){50}}
\put(50,0){\line(2,1){100}}
\put(50,0){\line(3,1){150}}
\put(50,0){\line(4,1){200}}
\put(100,0){\line(0,1){50}}
\put(100,0){\line(1,1){50}}
\put(100,0){\line(2,1){100}}
\put(100,0){\line(3,1){150}}
\put(150,0){\line(0,1){50}}
\put(150,0){\line(1,1){50}}
\put(150,0){\line(2,1){100}}
\put(200,0){\line(0,1){50}}
\put(200,0){\line(1,1){50}}
\put(250,0){\line(0,1){50}}

\put(50,50){\line(0,1){50}}
\put(50,50){\line(1,1){50}}
\put(50,50){\line(2,1){100}}
\put(50,50){\line(3,1){150}}
\put(50,50){\line(4,1){200}}
\put(100,50){\line(0,1){50}}
\put(100,50){\line(1,1){50}}
\put(100,50){\line(2,1){100}}
\put(100,50){\line(3,1){150}}
\put(150,50){\line(0,1){50}}
\put(150,50){\line(1,1){50}}
\put(150,50){\line(2,1){100}}
\put(200,50){\line(0,1){50}}
\put(200,50){\line(1,1){50}}
\put(250,50){\line(0,1){50}}

\put(50,100){\line(0,1){50}}
\put(50,100){\line(1,1){50}}
\put(50,100){\line(2,1){100}}
\put(50,100){\line(3,1){150}}
\put(50,100){\line(4,1){200}}
\put(100,100){\line(0,1){50}}
\put(100,100){\line(1,1){50}}
\put(100,100){\line(2,1){100}}
\put(100,100){\line(3,1){150}}
\put(150,100){\line(0,1){50}}
\put(150,100){\line(1,1){50}}
\put(150,100){\line(2,1){100}}
\put(200,100){\line(0,1){50}}
\put(200,100){\line(1,1){50}}
\put(250,100){\line(0,1){50}}

\put(50,150){\line(0,1){50}}
\put(50,150){\line(1,1){50}}
\put(50,150){\line(2,1){100}}
\put(50,150){\line(3,1){150}}
\put(50,150){\line(4,1){200}}
\put(100,150){\line(0,1){50}}
\put(100,150){\line(1,1){50}}
\put(100,150){\line(2,1){100}}
\put(100,150){\line(3,1){150}}
\put(150,150){\line(0,1){50}}
\put(150,150){\line(1,1){50}}
\put(150,150){\line(2,1){100}}
\put(200,150){\line(0,1){50}}
\put(200,150){\line(1,1){50}}
\put(250,150){\line(0,1){50}}

\end{picture}
\end{center}
\caption{The graph $H_{5,5}$}
\label{fig:H55}
\end{figure}
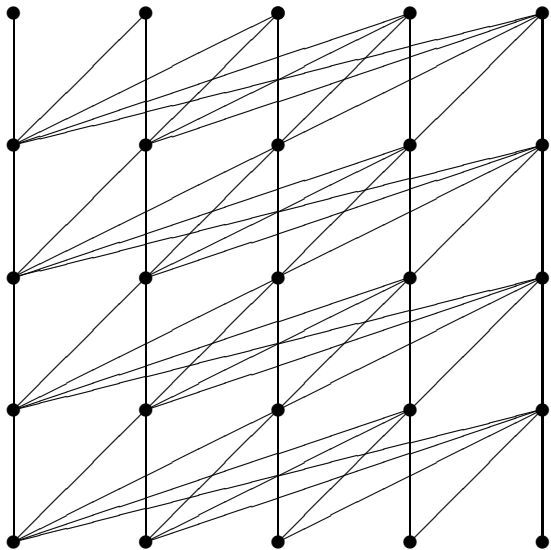

It is not difficult to see that the graph $H_{n,n}$ is a bipartite permutation graph. 
Moreover, it was proved in \cite{LR07} that $H_{n,n}$ is an $n$-universal bipartite permutation graph
in the sense that every bipartite permutation graph with $n$ vertices is an induced subgraph of $H_{n,n}$.
If a connected bipartite permutation graph is $P_k$-free, it occupies at most $k$ consecutive levels of $H_{n,n}$.
In other words, every connected $P_k$-free bipartite permutation graph is an induced subgraph of $H_{k,n}$.

In order to prove that $P_k$-free bipartite permutation graphs are well-quasi-ordered, we will show that
any connected graph in this class is a $k$-letter graph. This notion was introduced in \cite{Pet02} and 
its importance for our study is due to the following result also proved in \cite{Pet02}. 

\begin{thm}\label{thm:well}
For any fixed $k$, the class of $k$-letter graphs is well-quasi-ordered by
the induced subgraph relation. 
\end{thm}

The $k$-letter graphs have been characterized in \cite{Pet02} as follows.

\begin{thm}[Petkov\v sek \cite{Pet02}]\label{thm:kletter}
A graph $G=(V,E)$ is a $k$-letter graph if and only if
\begin{itemize}
\item[1.] there is a partition $V_1,\ldots,V_p$ of $V(G)$ with $p\le k$ such that each $V_i$ is either 
a clique or an independent set in $G$,
\item[2.] there is a linear ordering $L$ of $V(G)$ such that for each pair of indices $1\le i,j\le p$,
$i\ne j$, the intersection of $E$ with $V_i\times V_j$ is one of 
\begin{itemize}
\item[(a)] $L\cap (V_i\times V_j)$,
\item[(b)] $L^{-1}\cap (V_i\times V_j)$,
\item[(c)] $V_i\times V_j$,
\item[(d)] $\emptyset$.
\end{itemize}
\end{itemize}
\end{thm}

\begin{cor}\label{cor}
Connected $P_k$-free bipartite permutation graphs are $k$-letter graphs.
\end{cor} 
 
\begin{proof}
From Theorem~\ref{thm:kletter} it follows that an induced subgraph of a $k$-letter graph is again
a $k$-letter graph. In addition, we have seen already that any connected $P_k$-free bipartite permutation graph 
is an induced subgraph of $H_{k,n}$. Therefore, all we have to do is to prove that $H_{k,n}$ is a $k$-letter graph.
To this end, we define a partition $V_1,\ldots,V_k$ of the vertices of $H_{k,n}$ by defining $V_i$ to be the $i$-th row 
of $H_{k,n}$. Thus the first condition of Theorem~\ref{thm:kletter} is satisfied. Then we define a linear ordering $L$
of the vertices of $H_{k,n}$ by listing first the vertices of the first column consecutively from bottom to top, 
then the vertices of the second column, and so on. Now let's take any two subsets $V_i$ and $V_j$ with $i\ne j$.
If they are not consecutive rows of the graph, then the intersection of $E$ with $V_i\times V_j$ is empty.
If they are consecutive, then the intersection of $E$ with $V_i\times V_j$ is either $L\cap (V_i\times V_j)$
(if $i>j$) or $L^{-1}\cap (V_i\times V_j)$ (if $i<j$). Thus the second condition of Theorem~\ref{thm:kletter} 
is satisfied, which proves the corollary.
\end{proof}

Combining Corollary~\ref{cor} with Theorems~\ref{[H]} and~\ref{thm:well} we conclude that 

\begin{cor} 
For any fixed $k$, the class of $P_k$-free bipartite permutation graphs is well-quasi-ordered by the induced subgraph relation.   
\end{cor}

\end{document}